\documentclass[11pt]{amsart}
\usepackage{eurosym}
\usepackage{amssymb}
\usepackage{amsmath}
\usepackage{amsfonts}
\usepackage{graphicx,color}

\newtheorem{theorem}{Theorem}
\theoremstyle{plain}

\newtheorem{lemma}{Lemma}

\newtheorem{remark}{Remark}

\numberwithin{equation}{section}

\DeclareMathOperator{\esssup}{\rm ess\,sup}

\newcommand{\beq}{\begin{equation}}
\newcommand{\eeq}{\end{equation}}

\email{pasquale.candito@unirc.it}
\email{marano@dmi.unict.it}
\email{abdelkrim.moussaoui@univ-bejaia.dz}

\thanks{Work performed within the 2016--2018 Research Plan - Intervention Line 2: `Variational Methods and Differential Equations', and partially supported by GNAMPA of INDAM}

\begin{document}
\title{Nodal solutions to a Neumann problem for a class of $(p_1,p_2)$%
-Laplacian systems}
\author[P. Candito]{P. Candito}
\address[P. Candito]{DICEAM\\
Universit\`a degli Studi di Reggio Calabria \\
89100 Reggio Calabria, Italy}
\author[S.A. Marano]{S. A. Marano\dag}
\thanks{\dag Corresponding author}
\address[S.A. Marano]{Dipartimento di Matematica e Informatica \\
Universit\`a degli Studi di Catania \\
Viale A. Doria 6, 95125 Catania, Italy}
\author[A. Moussaoui]{A. Moussaoui}
\address[A. Moussaoui]{Biology Department \\
A. Mira Bejaia University \\
Targa Ouzemour, 06000 Bejaia, Algeria}

\begin{abstract}
Nodal solutions of a parametric $(p_1,p_2)$-Laplacian system, with Neumann
boundary conditions, are obtained by chiefly constructing appropriate
sub-super-solution pairs.
\end{abstract}

\maketitle

\section{Introduction}

Let $\Omega $ be a bounded domain in $\mathbb{R}^{N}$, $N\geq 2$, having a
smooth boundary $\partial \Omega $, let $f,g:\Omega \times \mathbb{R}%
^{2}\rightarrow \mathbb{R}$ be two Carath\'{e}odory functions, and let $%
1<p_{1},p_{2}<N$. Consider the Neumann, quasi-linear, parametric, elliptic
system 
\begin{equation}
\left\{ 
\begin{array}{ll}
\label{p}-\Delta _{p_{1}}u=f(x,u,v)+\lambda h_{1}(x) & \text{in}\;\;\Omega ,
\\ 
-\Delta _{p_{2}}v=g(x,u,v)+\lambda h_{2}(x) & \text{in}\;\;\Omega , \\ 
|\nabla u|^{p_{1}-2}\frac{\partial u}{\partial \eta }=|\nabla v|^{p_{2}-2}%
\frac{\partial v}{\partial \eta }=0 & \text{on}\;\;\partial \Omega.%
\end{array}
\right.  \tag{${\rm P}_{\lambda }$}
\end{equation}
Here, $\eta $ denotes the outward unit normal vector to $\partial \Omega $, $%
\Delta _{p_{i}}$ stands for the $p_{i}$-Laplace operator, i.e., 
\begin{equation*}
\Delta _{p_{i}}u:=\mathrm{div}(|\nabla u|^{p_{i}-2}\nabla u)\quad \forall
\,u\in W^{1,p_{i}}(\Omega ),
\end{equation*}
while $h_{i}\in L_{loc}^{\infty }(\Omega )$ exhibits both a singular
behavior and a change of sign near $\partial \Omega $. Precisely, we set 
\begin{equation}  \label{2}
h_{i}(x):=\mathrm{sgn}(d(x)-\delta )d(x)^{\gamma _{i}}=\left\{ 
\begin{array}{cc}
-d(x)^{\gamma _{i}} & \text{ if }d(x)<\delta , \\ 
d(x)^{\gamma _{i}} & \text{ if }d(x)>\delta ,%
\end{array}%
\right.
\end{equation}
where $0<\delta<\mathrm{diam}(\Omega)$, 
\begin{equation}  \label{16}
\gamma_{i}:=\lambda ^{-\theta p_{i}}(p_{i}-1)-1
\end{equation}
with $\lambda, \theta>0$ large enough, and 
\begin{equation}  \label{defd}
d(x):=\mathrm{dist}(x,\partial \Omega ),\quad x\in \overline{\Omega }.
\end{equation}
The pair $(u,v)\in W^{1,p_{1}}(\Omega )\times W^{1,p_{2}}(\Omega )$ is
called a (weak) solution to problem \eqref{p} provided 
\begin{equation}  \label{defsol}
\left\{ 
\begin{array}{l}
\int_{\Omega }|\nabla u|^{p_{1}-2}\nabla u\nabla \varphi \,\mathrm{d}x
=\int_{\Omega }\left( f(\cdot ,u,v)+\lambda h_{1}\right) \varphi \,\mathrm{d}%
x, \\ 
\phantom{} \\ 
\int_{\Omega }|\nabla v|^{p_{2}-2}\nabla v\nabla \psi \,\mathrm{d}x
=\int_{\Omega }\left( g(\cdot ,u,v)+\lambda h_{2}\right) \psi \,\mathrm{d}x%
\end{array}
\right.
\end{equation}
for all $(\varphi ,\psi )\in W^{1,p_{1}}(\Omega )\times W^{1,p_{2}}(\Omega )$%
. If $u,v$ are \emph{both} sign changing then we say that the solution $(u,v)
$ is nodal. Let us point out that, although $h_{i}$ $(i=1,2)$ is singular,
the integrals $\int_{\Omega }h_{1}\varphi \,\mathrm{d}x$ and $\int_{\Omega
}h_{2}\psi \,\mathrm{d}x$ in \eqref{defsol} take sense, because $-1<\gamma
_{i}<0$; see \eqref{54} below.

This paper establishes the existence of a \emph{nodal solution} of \eqref{p}%
, which turns out negative near $\partial\Omega$; cf. Theorem \ref{T1}. The
assumptions on $f$ and $g$ are $(\mathrm{h}_1)$--$(\mathrm{h}_2)$ in Section %
\ref{S4}. Roughly speaking, $(\mathrm{h}_1)$ requires a standard growth
rate, that makes finite the right-hand side of \eqref{defsol}, while $(%
\mathrm{h}_2)$ is a suitable condition at zero. We first construct a
sub-solution $(\underline{u},\underline{v})$, \emph{positive far from $%
\partial\Omega$}, and a super-solution $(\overline{u},\overline{v})$, \emph{%
negative near $\partial\Omega$}, such that $\underline{u} \leq\overline{u}$, 
$\underline{v}\leq\overline{v}\,$; see Lemma \ref{L3}. From a technical
point of view, it represents the most difficult part of the proof and is
performed by chiefly combining $(\mathrm{h}_2)$ with an auxiliary result
(Lemma \ref{L3}) based upon a nice property (Lemma \ref{L1}) of $C^1_0$%
-functions. After that, sub-super-solution and truncation arguments (cf.
Theorem \ref{T2}) yield the desired conclusion.

The question whether there exist \emph{positive solutions} to \eqref{p} is a
much simpler matter, which we address in Theorem \ref{T3}.

Dirichlet problems for elliptic systems have been thoroughly investigated
since some years, mainly via variational techniques \cite{MR,Ou},
sub-super-solution and truncation methods \cite{CaMo}, or fixed point
theorems \cite{IMP}. The paper \cite{deF} represents an attractive
introduction on the topic, but there is a wealth of good results and the
relavant literature looks daily increasing. For instance, new frameworks are:

\begin{itemize}
\item the existence of constant-sign solutions to \emph{singular} elliptic
systems, where nonlinearities possibly contain convection terms and/or
variable exponents appear \cite{AM,AMT,EPS,MMZ}.

\item the study of elliptic systems with equations driven by a $(p,q)$%
-Laplace like differential operator, i.e., $u\mapsto\Delta_p u+\mu\Delta_q u$%
, where $\mu\geq 0$ while $1<q<p<+\infty$; see \cite{MVV} and the references
therein.
\end{itemize}

As far as we know, much less attention has been paid to Neumann boundary
conditions: a quick search in the Mathematical Reviews shows that relevant
works are about a third of the total.

Surprisingly enough, excepting \cite{M,MMP,MZ}, where solutions with at
least one sign-changing component are obtained, so far we were not able to
find previous results concerning the existence of nodal solutions, neither
for the Dirichlet case nor for the Neumann one.

\section{Preliminaries}

Let $(X,\|\cdot\|)$ be a real Banach space and let $X^*$ be its topological
dual, with duality bracket $\langle\cdot, \cdot\rangle$. An operator $A:X\to
X^*$ is said to be:

\begin{itemize}
\item \emph{bounded} if it maps bounded sets into bounded sets.

\item \emph{coercive} provided $\displaystyle{\lim_{\Vert x\Vert\to+\infty}} 
\frac{\langle A(x),x\rangle}{\Vert x\Vert}=+\infty$.

\item \emph{pseudo-monotone} if $x_n\rightharpoonup x$ in $X$ and $%
\displaystyle{\limsup_{n\to+\infty}}\langle A(x_n),x_n-x\rangle\leq 0$ force 
$\displaystyle{\liminf_{n\to+\infty}}\langle A(x_n),x_n-z\rangle\geq\langle
A(x),x-z\rangle$ for all $z\in X$.

\item \emph{of type $(\mathrm{S})_+$} provided $x_n\rightharpoonup x$ in $X$
and $\displaystyle{\limsup_{n\to+\infty}}\langle A(x_n),x_n-x\rangle\leq 0$
imply $x_n\to x$ in $X$.
\end{itemize}

Recall (see, e.g., \cite[Theorem 2.99]{CLM}) that

\begin{theorem}
\label{brezis} If $X$ is reflexive and $A:X\to X^*$ is bounded, coercive,
and pseudo-monotone then $A(X)=X^*$.
\end{theorem}

Hereafter, $\Omega$ will denote a bounded domain of the real Euclidean $N$%
-space $(\mathbb{R}^N,|\cdot |)$, $N\geq 2$, with a $C^2$-boundary $%
\partial\Omega$, on which we will employ the $(N-1)$-dimensional Hausdorff
measure $\sigma $, while $\eta (x)$ indicates the outward unit normal vector
to $\partial\Omega $ at its point $x$. Given $\delta>0$, define 
\begin{equation*}
\Omega_{\delta }:=\{x\in \Omega:d(x)<\delta\}.
\end{equation*}
Write $|E|$ for the $N$-dimensional Lebesgue measure of the set $E\subseteq%
\mathbb{R}^{N}$. Let $u,v:\Omega \to\mathbb{R}$ and let $t\in\mathbb{R}$.
The symbol $u\leq v$ means $u(x)\leq v(x)$ for almost every $x\in \Omega$, 
\begin{equation*}
\Omega (u\leq v):=\{x\in \Omega :u(x)\leq v(x)\},\quad t_{\pm}:=\max \{\pm
t,0\},
\end{equation*}
and $r^{\prime }$ denotes the conjugate exponent of $r\in [1,+\infty]$.
Analogously one introduces $\Omega (u\geq v)$, etc. The Sobolev space $%
W^{1,r}(\Omega )$ will be equipped with the norm 
\begin{equation*}
\Vert u\Vert _{1,r}:=\left( \Vert u\Vert _{r}^{r}+\Vert \nabla
u\Vert_{r}^{r}\right) ^{\frac{1}{r}},\quad u\in W^{1,r}(\Omega ),
\end{equation*}
where, as usual, 
\begin{equation*}
\Vert v\Vert _{r}:=\left\{ 
\begin{array}{ll}
\left( \int_{\Omega }|v(x)|^{r}\mathrm{d}x\right) ^{\frac{1}{r}} & \text{ if 
}r<+\infty, \\ 
\phantom{} &  \\ 
\underset{x\in\Omega}{\esssup}\, |v(x)| & \text{ otherwise.}%
\end{array}
\right.
\end{equation*}
Moreover, 
\begin{equation*}
W_{+}^{1,r}(\Omega ) :=\{u\in W^{1,r}(\Omega ):0\leq u\},\quad
W_b^{1,r}(\Omega):=W^{1,r}(\Omega)\cap L^\infty(\Omega),
\end{equation*}
\begin{equation*}
[u,v ]:=\{ w\in W^{1,r}(\Omega): u\leq w\leq v\},\quad C^{1,\tau}_0(%
\overline{\Omega}):=\{ u\in C^{1,\tau}(\overline{\Omega}):
u\lfloor_{\partial\Omega}=0\}.
\end{equation*}
Let $d$ be as in \eqref{defd}, let $1<r<N$, and let $-r<s\leq 0$. It is
known that 
\begin{equation*}
\left(\int_\Omega d(x)^s |u(x)|^r\mathrm{d}x\right)^{\frac{1}{r}}\leq C\Vert
u\Vert _{1,r}\quad\forall\, u\in W^{1,r}(\Omega),
\end{equation*}
with suitable $C>0$; see \cite[Theorem 19.9, case (19.29)]{OK}. Accordingly,
by H\"{o}lder's inequality, if $-1<\beta\leq 0$ then 
\begin{equation}  \label{54}
\int_\Omega|d^\beta u|\, \mathrm{d}x\leq|\Omega|^{\frac{1}{r^{\prime }}}
\left(\int_\Omega d^{\beta r}|u|^r \mathrm{d}x\right)^{\frac{1}{r}} \leq
C|\Omega|^{\frac{1}{r^{\prime }}} \Vert u\Vert _{1,r},\;\; u\in
W^{1,r}(\Omega).
\end{equation}
Although the next auxiliary result is folklore, we shall make its proof.

\begin{lemma}
\label{L1} Suppose $u\in C^{1,\tau }_0(\overline{\Omega })$. Then there
exists $c>0$ such that 
\begin{equation}  \label{uoverd}
\big\Vert d^{-1} u\big\Vert_{C^{0,\frac{\tau}{\tau+1}}(\overline{\Omega}%
)}\leq c\Vert u\Vert_{C^{1,\tau}(\overline{\Omega})}.
\end{equation}
The constant $c$ does not depend on $u$.
\end{lemma}

\begin{proof}
First of all, observe that $u$ is Lipschitz continuous and one has 
\begin{equation}
|u(x)|\leq \mathrm{Lip}(u)\,d(x)\quad \forall \,x\in \overline{\Omega }.
\label{Lipu}
\end{equation}
The regularity of $\partial \Omega $ yields $\delta \in ]0,1[$, $\Pi \in
C^{1}(\Omega _{\delta },\partial \Omega )$ fulfilling 
\begin{equation}
d(x)=|x-\Pi (x)|,\;\;\frac{x-\Pi (x)}{|x-\Pi (x)|}=-\eta (\Pi (x)),\;\;]\Pi
(x),x]\subseteq \Omega ,\;\;x\in \Omega _{\delta }.  \label{propPi}
\end{equation}
To simplify notation, set $\sigma :=\frac{\tau }{\tau +1}$. Inequality %
\eqref{uoverd} easily follows once we achieve, for some $C_{1}:=C_{1}(%
\Omega)>0$, 
\begin{equation}
\sup \left\{ \frac{\Big\vert\frac{u(x)}{d(x)}-\frac{u(y)}{d(y)}\Big\vert}{%
|x-y|^{\sigma }}:x,y\in \Omega ,\, 0<|x-y|<\frac{\delta }{2}\right\} \leq
C_{1}\,\Vert u\Vert _{C^{1,\tau }(\overline{\Omega })}.  \label{found}
\end{equation}
So, pick $x,y\in \Omega $ such that $0<|x-y|<\frac{\delta }{2}$. If $%
\max\{d(x),d(y)\}\geq \delta $ then $x,y\in \Omega\setminus \Omega_{\delta
/2}$. Consequently, 
\begin{equation*}
\sup_{x\in\Omega\setminus \Omega_{\delta /2}}\Big\vert\nabla\frac{u(x)}{d(x)}%
\Big\vert \leq 2\,\frac{\mathrm{Lip}(u)}{\delta }+4\,\frac{\Vert
u\Vert_{\infty }}{\delta ^{2}} \leq \left( \frac{2}{\delta }+\frac{4}{\delta
^{2}}\right) \Vert u\Vert _{C^{1}(\overline{\Omega })},
\end{equation*}
because $d$ is 1-Lipschitz, and the Mean Value Theorem entails 
\begin{equation}  \label{casezero}
\frac{\Big\vert\frac{u(x)}{d(x)}-\frac{u(y)}{d(y)}\Big\vert}{|x-y|^{\sigma }}%
\leq C_{2}\Vert u\Vert _{C^{1}(\overline{\Omega })}.
\end{equation}
Assume now $d(y)\leq d(x)<\delta $; a similar argument applies when $%
d(x)\leq d(y)<\delta $. Two situations may occur. \newline
1) $d(x)\leq |x-y|^{\frac{1}{\tau +1}}$. Through the above-mentioned result
again, besides \eqref{propPi}, we obtain 
\begin{eqnarray*}
\frac{u(x)}{d(x)} &=&\frac{u(x)-u(\Pi (x))}{|x-\Pi (x)|}=-\nabla u(\hat{x}%
)\eta (\Pi (x)), \\
\frac{u(y)}{d(y)} &=&\frac{u(y)-u(\Pi (y))}{|y-\Pi (y)|}=-\nabla u(\hat{y}%
)\eta (\Pi (y))
\end{eqnarray*}
with appropriate $\hat{x}\in ]\Pi (x),x[$, $\hat{y}\in ]\Pi (y),y[$. This
immediately leads to 
\begin{equation*}
\begin{split}
\Big\vert\frac{u(x)}{d(x)}-\frac{u(y)}{d(y)}\Big\vert& \leq |\nabla u(\hat{x}%
)-\nabla u(\hat{y})| +|\nabla u(\hat{y})|\,|\eta (\Pi (x))-\eta (\Pi (y))| \\
& \leq \Vert u\Vert _{C^{1,\tau }(\overline{\Omega })}\left( |\hat{x}-\hat{y}%
|^{\tau } +\mathrm{Lip}(\eta )\mathrm{Lip}(\Pi )|x-y|\right) .
\end{split}%
\end{equation*}
On the other hand, 
\begin{equation*}
|\hat{x}-\hat{y}|\leq |\hat{x}-x|+|x-y|+|y-\hat{y}|\leq d(x)+|x-y|+d(y)\leq
3|x-y|^{\frac{1}{\tau +1}}
\end{equation*}
as $|x-y|<\frac{\delta }{2}<1$. Therefore, 
\begin{equation}
\Big\vert\frac{u(x)}{d(x)}-\frac{u(y)}{d(y)}\Big\vert\leq C_{3}\Vert
u\Vert_{C^{1,\tau }(\overline{\Omega })}|x-y|^{\sigma }.  \label{caseone}
\end{equation}
2) $d(x)>|x-y|^{\frac{1}{\tau +1}}$. Inequality \eqref{Lipu} gives 
\begin{equation}
\begin{split}
\Big\vert\frac{u(x)}{d(x)}-\frac{u(y)}{d(y)}\Big\vert& \leq \Big\vert\frac{%
u(x)-u(y)}{d(x)}\Big\vert +|u(y)|\Big\vert\frac{d(x)-d(y)}{d(x)d(y)}\Big\vert
\\
& \leq \mathrm{Lip}(u)\frac{|x-y|}{d(x)}+\mathrm{Lip}(u)\,d(y)\frac{|x-y|}{%
d(x)d(y)} \\
& \leq 2\mathrm{Lip}(u)|x-y|^{\sigma }\leq 2\Vert u\Vert _{C^{1,\tau }(%
\overline{\Omega })}|x-y|^{\sigma }.
\end{split}
\label{casetwo}
\end{equation}
Gathering together \eqref{casezero}--\eqref{casetwo} yields \eqref{found}
and completes the proof.
\end{proof}

Let $1<r<+\infty $. The operator $A_{r}:W^{1,r}(\Omega )\rightarrow
(W^{1,r}(\Omega ))^{\ast }$ defined by 
\begin{equation*}
\langle A_{r}(u),\varphi \rangle :=\int_{\Omega }|\nabla u|^{r-2}\nabla
u\nabla \varphi \,\mathrm{d}x\quad \forall \,u,\varphi \in W^{1,r}(\Omega )
\end{equation*}
stems from the negative $r$-Laplacian with homogeneous Neumann boundary
conditions. Proposition 1 in \cite{MMP2011} ensures that it is of type $(%
\mathrm{S})_{+}$ while, taking \cite[Remark 8]{MMT} into account, if $u\in
W^{1,r}(\Omega )\cap L^{\infty }(\Omega )$, $w\in L^{\infty }(\Omega )$, and 
\begin{equation*}
\langle A_{r}(u),\varphi \rangle =\int_{\Omega }w(x)\varphi (x)\,\mathrm{d}%
x\quad \forall \,\varphi \in W^{1,r}(\Omega )
\end{equation*}
then $u\in C^{1,\tau }(\overline{\Omega })$, with suitable $\tau \in ]0,1[$,
as well as $\frac{\partial u}{\partial \eta }=0$ on $\partial \Omega $.

Denote by $\lambda _{1,r}$ the first eigenvalue of $-\Delta_r$ in $%
W^{1,r}_0(\Omega)$. It is known \cite{Le} that $\lambda_{1,r}$ possesses a
unique eigenfunction $\varphi _{1,r}$ enjoying the properties below.

\begin{itemize}
\item $\varphi_{1,r}\in\mathrm{int}(C_+)$, where $C_+:=\{ u\in C^1_0(%
\overline{\Omega}): u\geq 0\}$.

\item $\Vert\varphi_{1,r}\Vert_r=1$.

\item Any other eigenfunction turns out to be a scalar multiple of $%
\varphi_{1,r}$.
\end{itemize}

Finally, we say that $j:\Omega\times\mathbb{R}^2\to\mathbb{R}$ is a Carath%
\'{e}odory function provided

\begin{itemize}
\item $x\mapsto j(x,s,t)$ is measurable for every $(s,t)\in\mathbb{R}^2$, and

\item $(s,t)\mapsto j(x,s,t)$ is continuous for almost all $x\in \Omega$.
\end{itemize}

\section{A sub-super-solution theorem}

\label{S3} This section investigates the existence of solutions to \eqref{p}
without sign information. Recall that $f,g:\Omega\times\mathbb{R}^2\to%
\mathbb{R}$ satisfy Carath\'{e}odory's conditions. The following assumptions
will be posited.

\begin{itemize}
\item[$(\mathrm{a}_{1})$] For every $\rho>0$ there exists $M_\rho>0$ such
that 
\begin{equation*}
\max \{|f(x,s,t)|,|g(x,s,t)|\}\leq M_\rho\quad\text{in}\quad\Omega \times[%
-\rho,\rho]^{2}.
\end{equation*}

\item[$(\mathrm{a}_{2})$] With appropriate $(\underline{u},\underline{v}),(%
\overline{u},\overline{v})\in W_{b}^{1,p_{1}}(\Omega )\times
W_{b}^{1,p_{2}}(\Omega )$ one has $\underline{u}\leq \overline{u}$, $%
\underline{v}\leq \overline{v}$, as well as 
\begin{equation}
\left\{ 
\begin{array}{l}
\int_{\Omega }|\nabla \underline{u}|^{p_{1}-2}\nabla \underline{u}%
\nabla\varphi \,\mathrm{d}x -\int_{\Omega }\left( f(\cdot ,\underline{u}%
,v)+\lambda h_{1}\right) \varphi \,\mathrm{d}x\leq 0, \\ 
\phantom{} \\ 
\int_{\Omega }|\nabla \underline{v}|^{p_{2}-2}\nabla \underline{v}\,\nabla
\psi \,\mathrm{d}x -\int_{\Omega }\left( g(\cdot ,u,\underline{v})+\lambda
h_{2}\right) \psi \,\mathrm{d}x\leq 0,%
\end{array}
\right.  \label{c2}
\end{equation}
\begin{equation}
\left\{ 
\begin{array}{l}
\int_{\Omega }|\nabla \overline{u}|^{p_{1}-2}\nabla \overline{u}%
\,\nabla\varphi \,\mathrm{d}x -\int_{\Omega }\left( f(\cdot ,\overline{u}%
,v)+\lambda h_{1}\right) \varphi \,\mathrm{d}x\geq 0, \\ 
\phantom{} \\ 
\int_{\Omega }|\nabla \overline{v}|^{p_{2}-2}\nabla \overline{v}\,\nabla\psi
\,\mathrm{d}x -\int_{\Omega }\left( g(\cdot ,u,\overline{v})+\lambda
h_{2}\right) \psi \,\mathrm{d}x\geq 0%
\end{array}%
\right.  \label{c3}
\end{equation}
for all $(\varphi ,\psi )\in W_{+}^{1,p_{1}}(\Omega )\times
W_{+}^{1,p_{2}}(\Omega )$, $(u,v)\in W^{1,p_{1}}(\Omega) \times
W^{1,p_{2}}(\Omega )$ such that $(u,v)\in [\underline{u},\overline{u}]\times
[\underline{v},\overline{v}]$.
\end{itemize}

Under $(\mathrm{a}_{1})$, the above integrals involving $f$ and $g$ take
sense, because $\underline{u},\underline{v},\overline{u},\overline{v}$ are
bounded.

\begin{theorem}
\label{T2} Suppose $(\mathrm{a}_{1})$--$(\mathrm{a}_{2})$ hold true. Then,
for every $\lambda \geq 0$, problem \eqref{p} possesses a solution $(u,v)\in
W_{b}^{1,p_{1}}(\Omega )\times W_{b}^{1,p_{2}}(\Omega )$ such that 
\begin{equation}  \label{19}
\underline{u}\leq u\leq \overline{u}\quad \text{and}\quad \underline{v}\leq
v\leq \overline{v}.
\end{equation}
If $\lambda =0$ then $(u,v)\in C^{1,\tau }(\overline{\Omega })\times
C^{1,\tau }(\overline{\Omega })$ with suitable $\tau \in\ ]0,1[$. Moreover, $%
\frac{\partial u}{\partial \eta }=\frac{\partial v}{\partial \eta }=0$ on $%
\partial \Omega $.
\end{theorem}

\begin{proof}
Given $u\in W^{1,p_{1}}(\Omega )$, $v\in W^{1,p_{2}}(\Omega )$, we define 
\begin{equation*}
T_{1}(u):=\left\{ 
\begin{array}{ll}
\underline{u} & \text{when }u\leq \underline{u}, \\ 
u & \text{if }\underline{u}\leq u\leq \overline{u}, \\ 
\overline{u} & \text{otherwise,}%
\end{array}
\right. \quad T_{2}(v):=\left\{ 
\begin{array}{ll}
\underline{v} & \text{when }v\leq \underline{v}, \\ 
v & \text{if }\underline{v}\leq v\leq \overline{v}, \\ 
\overline{v} & \text{otherwise.}%
\end{array}
\right.
\end{equation*}
Lemma 2.89 of \cite{CLM} ensures that the functions $T_{i}:W^{1,p_{i}}(%
\Omega )\rightarrow W^{1,p_{i}}(\Omega )$, $i=1,2$, are continuous and
bounded. If $\rho >0$ satisfies 
\begin{equation*}
-\rho \leq \underline{u}\leq \overline{u}\leq \rho ,\quad -\rho \leq%
\underline{v}\leq \overline{v}\leq \rho ,
\end{equation*}
while $\mathcal{N}_{f}$ (resp., $\mathcal{N}_{g}$) denotes the Nemitski
operators associated with $f$ (resp., $g$) then, thanks to $(\mathrm{a}_{1})$%
, the maps 
\begin{equation}
N_{f}\circ (T_{1},T_{2}):W^{1,p_{1}}(\Omega )\times
W^{1,p_{2}}(\Omega)\rightarrow L^{p_{1}^{\prime }}(\Omega )\hookrightarrow
W^{-1,p_{1}}(\Omega),  \label{NF}
\end{equation}
\begin{equation}
N_{g}\circ (T_{1},T_{2}):W^{1,p_{1}}(\Omega )\times
W^{1,p_{2}}(\Omega)\rightarrow L^{p_{2}^{\prime }}(\Omega )\hookrightarrow
W^{-1,p_{2}}(\Omega)  \label{NG}
\end{equation}
enjoy the same property. Moreover, setting 
\begin{equation*}
\chi _{1}(x,s):=-(\underline{u}(x)-s)_{+}^{p_{1}-1}+(s-\overline{u}%
(x))_{+}^{p_{1}-1},\quad (x,s)\in \Omega \times \mathbb{R},
\end{equation*}
\begin{equation*}
\chi _{2}(x,t):=-(\underline{v}(x)-t)_{+}^{p_{2}-1}+(t-\overline{v}%
(x))_{+}^{p_{2}-1},\quad (x,t)\in \Omega \times \mathbb{R},
\end{equation*}
one has 
\begin{equation}
\int_{\Omega }\chi _{1}(\cdot ,u)u\,\mathrm{d}x\geq C_{1}\Vert
u\Vert_{p_{1}}^{p_{1}}-C_{2}\quad \forall \,u\in W^{1,p_{1}}(\Omega ),
\label{35}
\end{equation}
\begin{equation}
\int_{\Omega }\chi _{2}(\cdot ,v)v\,\mathrm{d}x\geq C_{1}^{\prime }\Vert
v\Vert _{p_{2}}^{p_{2}}-C_{2}^{\prime }\quad \forall \,v\in
W^{1,p_{2}}(\Omega )  \label{36}
\end{equation}
with appropriate constants $C_{i},C_{i}^{\prime }>0$; see, e.g., \cite[pp.
95--96]{CLM}. Penalties $\chi _{1}$ and $\chi _{2}$ are involved in the
following auxiliary problem: 
\begin{equation}
\left\{ 
\begin{array}{ll}
-\Delta _{p_{1}}{u}=f_{\mu }(x,u,v) & \text{in}\;\;\Omega , \\ 
-\Delta _{p_{2}}{v}=g_{\mu }(x,u,v) & \text{in}\;\;\Omega , \\ 
|\nabla u|^{p_{1}-2}\frac{\partial u}{\partial \eta }=|\nabla v|^{p_{2}-2}%
\frac{\partial v}{\partial \eta }=0 & \text{on}\;\; \partial \Omega ,%
\end{array}
\right.  \label{P*1}
\end{equation}
where, for every $\mu >0$, $(u,v)\in W^{1,p_{1}}(\Omega )\times
W^{1,p_{2}}(\Omega )$, 
\begin{equation*}
f_{\mu }(\cdot ,u,v):=f(\cdot ,T_{1}(u),T_{2}(v))+\lambda h_{1}-\mu
\chi_{1}(\cdot ,u),
\end{equation*}
\begin{equation*}
g_{\mu }(\cdot ,u,v):=g(\cdot ,T_{1}(u),T_{2}(v))+\lambda h_{2}-\mu
\chi_{2}(\cdot ,v).
\end{equation*}
Evidently, 
\begin{equation*}
\left\{ 
\begin{array}{l}
f_{\mu }(\cdot ,u,v)=f(\cdot ,u,v)+\lambda h_{1}, \\ 
g_{\mu }(\cdot ,u,v)=g(\cdot ,u,v)+\lambda h_{2},%
\end{array}
\right.  \label{21}
\end{equation*}
once $(u,v)\in W^{1,p_{1}}(\Omega )\times W^{1,p_{2}}(\Omega )$ satisfies %
\eqref{19}.

Let $\mathcal{E}$ be the space $W^{1,p_{1}}(\Omega )\times
W^{1,p_{2}}(\Omega )$ equipped with the norm 
\begin{equation*}
\Vert (u,v)\Vert _{\mathcal{E}}:=\Vert u\Vert _{1,p_{1}}+\Vert
v\Vert_{1,p_{2}},\quad (u,v)\in \mathcal{E},
\end{equation*}
and let $\mathcal{B}_{\mu }:\mathcal{E}\rightarrow \mathcal{E}^{\prime }$ be
defined by 
\begin{equation*}
\begin{split}
\left\langle\mathcal{B}_{\mu}(u,v),(\varphi,\psi )\right\rangle:= &
\int_{\Omega }(|\nabla u|^{p_{1}-2}\nabla u\nabla\varphi +|\nabla
v|^{p_{2}-2}\nabla v\nabla \psi )\,\mathrm{d}x \\
& -\int_{\Omega}f_{\mu }(\cdot,u,v)\varphi \,\mathrm{d}x-\int_{\Omega
}g_{\mu}(\cdot,u,v)\psi \,\mathrm{d}x
\end{split}%
\end{equation*}
for all $(u,v),(\varphi ,\psi )\in \mathcal{E}$. We shall verify that $%
\mathcal{B}_{\mu }$ fulfills the assumptions of Theorem \ref{brezis}
provided $\mu $ is large enough. To this end, observe at first that %
\eqref{54} entails 
\begin{equation}
\int_{\Omega }|h_{1}\varphi |\,\mathrm{d}x\leq C_{3}\Vert \varphi \Vert
_{1,p_{1}}\,,\quad \int_{\Omega }|h_{2}\psi |\,\mathrm{d}x\leq C_{3}^{\prime
}\Vert \psi \Vert _{1,p_{2}}  \label{52}
\end{equation}
because $-1<\gamma _{i}<0$.  \vskip3pt 1) $\mathcal{B}_{\mu }$ \textit{is
continuous.} \vskip3pt \noindent Suppose $(u_{n},v_{n})\rightarrow (u,v)$ in 
$\mathcal{E}$. Pick any $(\varphi ,\psi )\in \mathcal{E}$ such that $%
\Vert(\varphi ,\psi )\Vert _{\mathcal{E}}\leq 1$. If $p_{1},p_{2}\geq 2$
then, through \cite[Lemma 5.3]{GM} besides H\"{o}lder's inequality, one
easily obtains 
\begin{equation}
\begin{split}
& \int_{\Omega }\left\vert \left\langle |\nabla u_{n}|^{p_{1}-2}\nabla
u_{n}-|\nabla u|^{p_{1}-2}\nabla u,\nabla \varphi \right\rangle \right\vert 
\mathrm{d}x \\
& +\int_{\Omega }\left\vert \left\langle |\nabla v_{n}|^{p_{2}-2}\nabla
v_{n}-|\nabla v|^{p_{2}-2}\nabla v,\nabla \psi \right\rangle \right\vert 
\mathrm{d}x \\
& \leq c_{p_{1}}\left\Vert \nabla u_{n}+\nabla u\right\Vert
_{p_{1}}^{p_{1}^{\prime }(p_{1}-2)}\left\Vert u_{n}-u\right\Vert
_{1,p_{1}}^{p_{1}^{\prime }} \\
& +c_{p_{2}}\left\Vert \nabla v_{n}+\nabla v\right\Vert
_{p_{2}}^{p_{2}^{\prime }(p_{2}-2)}\left\Vert v_{n}-v\right\Vert
_{1,p_{2}}^{p_{2}^{\prime }}.
\end{split}
\label{38}
\end{equation}%
The case $1<p_{1},p_{2}\leq 2$ carries over via \cite[Lemma 5.4]{GM}, with
the right-hand side of \eqref{38} replaced by 
\begin{equation*}
c_{p_{1}}^{\prime }\left\Vert u_{n}-u\right\Vert
_{1,p_{1}}^{p_{1}-1}+c_{p_{2}}^{\prime }\left\Vert v_{n}-v\right\Vert
_{1,p_{2}}^{p_{2}-1},
\end{equation*}%
while the remaining situations are analogous. A simple argument based on the
Dominated Convergence Theorem, besides the continuity of maps \eqref{NF}, %
\eqref{NG}, and 
\begin{equation}
w\in W^{1,p_{i}}(\Omega )\mapsto \chi _{i}(\cdot ,w)\in L^{p_{i}^{\prime
}}(\Omega )\hookrightarrow W^{-1,p_{i}}(\Omega ),  \label{gammai}
\end{equation}%
shows that 
\begin{equation}
\lim_{n\rightarrow +\infty }\int_{\Omega }|(f_{\mu }(\cdot
,u_{n},v_{n})-f_{\mu }(\cdot ,u,v))\varphi |\mathrm{d}x=0  \label{40}
\end{equation}%
as well as 
\begin{equation}
\lim_{n\rightarrow +\infty }\int_{\Omega }|(g_{\mu }(\cdot
,u_{n},v_{n})-g_{\mu }(\cdot ,u,v))\psi |\mathrm{d}x=0.  \label{41}
\end{equation}%
Finally, since 
\begin{equation*}
\begin{split}
& \left\vert \langle \mathcal{B}_{\mu }(u_{n},v_{n})-\mathcal{B}%
_{\mu}(u,v),(\varphi ,\psi )\rangle \right\vert \\
& \leq \int_{\Omega }\left\vert \langle |\nabla u_{n}|^{p_{1}-2}\nabla
u_{n}-|\nabla u|^{p_{1}-2}\nabla u,\nabla \varphi \rangle \right\vert\mathrm{%
d}x \\
& +\int_{\Omega }\left\vert \left\langle |\nabla v_{n}|^{p_{2}-2}\nabla
v_{n}-|\nabla v|^{p_{2}-2}\nabla v,\nabla \psi \right\rangle \right\vert 
\mathrm{d}x \\
& +\int_{\Omega }|f_{\mu }(\cdot ,u_{n},v_{n})-f_{\mu }(\cdot ,u,v)||\varphi|%
\mathrm{d}x \\
& +\int_{\Omega }|g_{\mu }(\cdot ,u_{n},v_{n})-g_{\mu }(\cdot ,u,v)||\psi |%
\mathrm{d}x
\end{split}%
\end{equation*}
for all $n\in \mathbb{N}$, \eqref{38}--\eqref{41} easily produce $\Vert 
\mathcal{B}_{\mu }(u_{n},v_{n})-\mathcal{B}_{\mu }(u,v)\Vert _{\mathcal{E}%
^{\prime }}\rightarrow 0$. \vskip3pt 2) $\mathcal{B}_{\mu }$ \textit{is
bounded.} \vskip3pt \noindent It immediately follows from \eqref{52} and the
boundedness of maps \eqref{NF}, \eqref{NG}, \eqref{gammai}. \vskip3pt 3) $%
\mathcal{B}_{\mu }$ \textit{is coercive.} \vskip3pt \noindent Using %
\eqref{52} with $\varphi :=u$ and $\psi :=v$ yields 
\begin{equation*}
\int_{\Omega }|h_{1}u|\,\mathrm{d}x\leq C_{3}\Vert u\Vert _{1,p_{1}}\,,\quad
\int_{\Omega }|h_{2}v|\,\mathrm{d}x\leq C_{3}^{\prime }\Vert v\Vert
_{1,p_{2}}.
\end{equation*}
Hence, by $(\mathrm{a}_{1})$, 
\begin{equation}
\int_{\Omega }|f_{\mu }(\cdot ,u,v)u|\,\mathrm{d}x\leq M_{\rho }C_{4}\Vert
u\Vert _{p_{1}}+\lambda C_{3}\Vert u\Vert _{1,p_{1}},  \label{new2}
\end{equation}%
\begin{equation}
\int_{\Omega }|g_{\mu }(\cdot ,u,v)v|\,\mathrm{d}x\leq M_{\rho
}C_{4}^{\prime }\Vert v\Vert _{p_{2}}+\lambda C_{3}^{\prime }\Vert v\Vert
_{1,p_{2}}\,.  \label{new3}
\end{equation}%
Via \eqref{new2}--\eqref{new3} and \eqref{35}--\eqref{36} we thus arrive at 
\begin{equation*}
\begin{array}{l}
\left\langle \mathcal{B}_{\mu }(u,v),(u,v)\right\rangle \geq \Vert \nabla
u\Vert _{p_{1}}^{p_{1}}+\Vert \nabla v\Vert _{p_{2}}^{p_{2}}+\mu C_{1}^{\ast
}(\Vert u\Vert _{p_{1}}^{p_{1}}+\Vert v\Vert _{p_{2}}^{p_{2}}) \\ 
\phantom{} \\ 
-M_{\rho }C_{4}^{\ast }(\Vert u\Vert _{p_{1}}+\Vert v\Vert _{p_{2}})-\lambda
C_{3}^{\ast }(\Vert u\Vert _{1,p_{1}}+\Vert v\Vert _{1,p_{2}})-\mu
(C_{2}+C_{2}^{\prime }),%
\end{array}%
\end{equation*}%
where $C_{1}^{\ast }:=\min \{C_{1},C_{1}^{\prime }\}$, $C_{3}^{\ast }:=\max
\{C_{3},C_{3}^{\prime }\}$, $C_{4}^{\ast }:=\max \{C_{4},C_{4}^{\prime }\}$.
This inequality forces 
\begin{equation*}
\lim_{n\rightarrow +\infty }\frac{\langle \mathcal{B}_{\mu
}(u_{n},v_{n}),(u_{n},v_{n})\rangle }{\Vert (u_{n},v_{n})\Vert _{\mathcal{E}}%
}=+\infty ,
\end{equation*}%
as desired. \vskip3pt 4) $\mathcal{B}_{\mu }$ \textit{is pseudo-monotone.} %
\vskip3pt \noindent Suppose $(u_{n},v_{n})\rightharpoonup (u,v)$ in $%
\mathcal{E}$, 
\begin{equation}
\limsup_{n\rightarrow +\infty }\langle \mathcal{B}_{%
\mu}(u_{n},v_{n}),(u_{n},v_{n})-(u,v)\rangle \leq 0,  \label{B1}
\end{equation}
and, without loss of generality, 
\begin{equation}  \label{20}
(u_n,v_n)\in [\underline{u},\overline{u}]\times [\underline{v},\overline{v}%
]\quad\forall\, n\in\mathbb{N}.
\end{equation}
Since the maps \eqref{gammai} are completely continuous, exploiting $(%
\mathrm{a}_{1})$, \eqref{20}, \eqref{54} (recall that $-1<\gamma _{i}<0$),
and the Dominated Convergence Theorem, one has 
\begin{equation*}
\begin{split}
\lim_{n\rightarrow +\infty }\int_{\Omega }f_{\mu
}(\cdot,u_{n},v_{n})(u_{n}-u)\,\mathrm{d}x& =0, \\
\lim_{n\rightarrow +\infty }\int_{\Omega }g_{\mu
}(\cdot,u_{n},v_{n})(v_{n}-v)\,\mathrm{d}x& =0,
\end{split}%
\end{equation*}
which, when combined with \eqref{B1}, lead to 
\begin{equation}
\limsup_{n\rightarrow +\infty }[\left\langle
A_{p_{1}}(u_{n}),u_{n}-u\right\rangle +\left\langle
A_{p_{2}}(v_{n}),v_{n}-v\right\rangle ]\leq 0.  \label{new}
\end{equation}
Through standard results we achieve 
\begin{equation}
\lim_{n\rightarrow +\infty }\left\langle A_{p_{1}}(u),u_{n}-u\right\rangle
=\lim_{n\rightarrow +\infty }\left\langle A_{p_{2}}(v),v_{n}-v\right\rangle
=0,  \label{new1}
\end{equation}
so that \eqref{new} becomes 
\begin{equation*}
\limsup_{n\rightarrow +\infty }[\left\langle
A_{p_{1}}(u_{n})-A_{p_{1}}(u),u_{n}-u\right\rangle +\left\langle
A_{p_{2}}(v_{n})-A_{p_{2}}(v),v_{n}-v\right\rangle ]\leq 0.
\end{equation*}
By monotonicity, it actually means 
\begin{equation*}
\lim_{n\rightarrow +\infty }\left\langle
A_{p_{1}}(u_{n})-A_{p_{1}}(u),u_{n}-u\right\rangle =\lim_{n\rightarrow
+\infty }\left\langle A_{p_{2}}(v_{n})-A_{p_{2}}(v),v_{n}-v\right\rangle =0
\end{equation*}
Now, use \eqref{new1} and recall that $A_{p_{i}}$ is of type $(\mathrm{S}%
)_{+}$ to get $(u_{n},v_{n})\rightarrow (u,v)$ in $\mathcal{E}$, whence 
\begin{equation*}
\lim_{n\rightarrow +\infty }\langle \mathcal{B}_{%
\mu}(u_{n},v_{n}),(u_{n},v_{n})-(\varphi ,\psi )\rangle =\langle \mathcal{B}%
_{\mu }(u,v),(u,v)-(\varphi ,\psi )\rangle \;\;\forall \,(\varphi ,\psi )\in 
\mathcal{E},
\end{equation*}
because $\mathcal{B}_{\mu }$ is continuous. \vskip3pt At this point, Theorem %
\ref{brezis} can be applied. Therefore, there exists $(u,v)\in \mathcal{E}$
fulfilling 
\begin{equation}
\left\langle \mathcal{B}_{\mu }(u,v),(\varphi ,\psi )\right\rangle
=0,\;\;(\varphi ,\psi )\in \mathcal{E}.  \label{ws}
\end{equation}
Moreover, due to \cite[Theorem 3]{CF}, one has 
\begin{equation*}
|\nabla u|^{p_{1}-2}\frac{\partial u}{\partial \eta }=|\nabla v|^{p_{2}-2} 
\frac{\partial v}{\partial \eta }=0\;\;\text{on }\partial \Omega .
\end{equation*}
Thus, $(u,v)$ is a weak solution of \eqref{P*1}. Let us next verify that
inequalities \eqref{19} hold true. Writing \eqref{ws} for $(\varphi ,\psi
):=((u-\overline{u})_{+},0)$ and taking \eqref{c3} into account, we infer 
\begin{equation*}
\begin{split}
& \int_{\Omega }|\nabla u|^{p_{1}-2}\nabla u\,\nabla (u-\overline{u})_{+}\,%
\mathrm{d}x =\int_{\Omega }f_{\mu }(\cdot ,u,v)(u-\overline{u})_{+}\,\mathrm{%
d}x \\
& =\int_{\Omega}\hskip-3pt f(\cdot ,T_{1}u,T_{2}v)(u-\overline{u})_{+}%
\mathrm{d}x +\lambda\hskip-1pt \int_{\Omega }\hskip-3pt h_{1}(u-\overline{u}%
)_{+}\mathrm{d}x -\mu\hskip-1pt\int_{\Omega }\hskip-3pt\chi _{1}(\cdot ,u)(u-%
\overline{u})_{+}\mathrm{d}x \\
& =\int_{\Omega }f(\cdot ,\overline{u},T_{2}v)(u-\overline{u})_{+}\mathrm{d}%
x +\lambda \int_{\Omega }h_{1}(u-\overline{u})_{+}\mathrm{d}%
x-\mu\int_{\Omega }(u-\overline{u})_{+}^{p_{1}}\mathrm{d}x \\
& \leq \int_{\Omega }|\nabla \overline{u}|^{p_{1}-2}\nabla \overline{u}%
\,\nabla (u-\overline{u})_{+}\mathrm{d}x -\mu \int_{\Omega }(u-\overline{u}%
)_{+}^{p_{1}}\mathrm{d}x,
\end{split}%
\end{equation*}
namely 
\begin{equation*}
\int_{\Omega }\left( |\nabla u|^{p_{1}-2}\nabla u-|\nabla \overline{u}%
|^{p_{1}-2}\nabla \overline{u}\right) \nabla (u-\overline{u})_{+}\,\mathrm{d}%
x\leq -\mu \int_{\Omega }(u-\overline{u})_{+}^{p_{1}}\mathrm{d}x\leq 0.
\end{equation*}
The monotonicity of $A_{p_{1}}$ directly yields $u\leq \overline{u}$. To see
that $\underline{u}\leq u$, pick $(\varphi ,\psi ):=((\underline{u}-u)_{+},0)
$ and employ \eqref{c2}. A quite similar reasoning then gives $\underline{v}%
\leq v\leq \overline{v}$. Consequently, $(u,v)$ is a solution of \eqref{p}
within $[\underline{u},\overline{u}]\times \lbrack \underline{v},\overline{v}%
]$.

Finally, let $\lambda =0$. Arguing exactly as in \cite[Remark 8]{MMT} we
obtain here $(u,v)\in C^{1,\tau }(\overline{\Omega })\times C^{1,\tau }(%
\overline{\Omega })$ for some $\tau \in ]0,1[$ and $\frac{\partial u}{%
\partial \eta }=\frac{\partial v}{\partial \eta }=0$ on $\partial \Omega $,
which completes the proof.
\end{proof}

\begin{remark}
The conclusion of Theorem \ref{T2} remains true if we replace Neumann
boundary conditions with Dirichlet ones.
\end{remark}

\begin{remark}
Hypothesis $(\mathrm{a}_2)$ will be summarized saying that $(\underline{u},%
\underline{v})$ and $(\overline{u},\overline{v})$ represent a sub-solution
and a super-solution pair, respectively, for \eqref{p}.
\end{remark}

\section{Existence of solutions}

\label{S4}

Our first goal is to construct sub- and super-solution pairs of \eqref{p}.
With this aim, consider the homogeneous Dirichlet problem 
\begin{equation}  \label{5bis}
\left\{ 
\begin{array}{ll}
-\Delta _{p_{i}}u=1 & \text{ in }\Omega , \\ 
u=0 & \text{ on }\partial \Omega ,%
\end{array}
\right.
\end{equation}
$i=1,2$, which admits a unique solution $z_{i}\in C^{1,\tau}_0(\overline{%
\Omega})$.

\begin{lemma}
There exist $\hat L,l,L>0$ such that 
\begin{equation}  \label{12}
\Vert \nabla z_{i}\Vert _{\infty }\leq \hat{L},
\end{equation}
\begin{equation}  \label{12*}
ld\leq z_{i}\leq Ld\;\;\text{in}\;\;\Omega,\;\;\text{and}\;\;\frac{\partial
z_{i}}{\partial \eta }<0\;\;\text{on}\;\;\partial \Omega ,
\end{equation}
\end{lemma}

\begin{proof}
Theorem 3.1 of \cite{CM} ensures that \eqref{12} holds. The Strong Maximum
Principle entails $ld\leq z_i$, for appropriate $l>0$, as well as $\frac{%
\partial z_{i}}{\partial \eta }\lfloor_{\partial\Omega}<0$. Since $%
\partial\Omega$ is smooth, we can find $\delta>0$ and $\Pi\in
C^1(\Omega_\delta,\partial\Omega)$ satisfying \eqref{propPi}. Thus, the Mean
Value Theorem, when combined with \eqref{12}, lead to 
\begin{equation}  \label{MVT1}
|z_i(x)|=|z_i(x)-z_i(\Pi(x))|\leq\hat L |x-\Pi(x)|=\hat L d(x)\quad\forall\,
x\in\Omega_\delta\, .
\end{equation}
Define 
\begin{equation*}
L:=\max\left\{\hat L,\, \max_{\Omega\setminus\Omega_\delta}\frac{z_i}{d}\,
,\, i=1,2\right\}.
\end{equation*}
On account of \eqref{MVT1}, one evidently has $z_i\leq Ld$.
\end{proof}

Now, given $\delta >0$, denote by $z_{i,\delta }\in C^{1,\tau}_0(\overline{%
\Omega})$ the solution of the Dirichlet problem 
\begin{equation}
-\Delta _{p_{i}}u=\left\{ 
\begin{array}{ll}
1 & \text{if }x\in \Omega \backslash \overline{\Omega }_{\delta }, \\ 
-\lambda ^{\theta p_{i}}d(x)^{\gamma _{i}} & \text{otherwise},%
\end{array}
\right. \quad u=0\text{ on }\partial \Omega,  \label{1}
\end{equation}
where $i=1,2$, 
\begin{equation}
\theta >1+p_{i}^{\prime }>1+\frac{1}{p_{i}-1},  \label{5}
\end{equation}
while $\gamma_i$ is as in \eqref{16} for $\lambda,\theta>0$ big enough.
Existence and uniqueness directly stem from Minty-Browder's Theorem, because 
$-1<\gamma_i<0$ forces $d^{\gamma_i}\in W^{-1,p_i^{\prime }}(\Omega)$; see %
\eqref{54}.

\begin{lemma}
\label{L3} If $\delta>0$ is small enough then

\begin{itemize}
\item[$(\mathrm{j}_1)$] $\frac{\partial z_{i,\delta}}{\partial\eta}<\frac{1}{%
2}\frac{\partial z_i}{\partial\eta}<0$ on $\partial \Omega$, and

\item[$(\mathrm{j}_2)$] $z_{i,\delta}\geq\frac{1}{2}\, z_{i}$ in $\Omega$.
\end{itemize}
\end{lemma}

\begin{proof}
Let $\hat M_{i}>0$ fulfill 
\begin{equation}
\Vert z_{i}\Vert _{C^{1,\tau }(\overline{\Omega })}\leq\hat M_{i},\quad
\Vert z_{i,\delta }\Vert _{C^{1,\tau }(\overline{\Omega })}\leq\hat
M_{i}\,,\quad\delta >0.  \label{emme}
\end{equation}
Using \eqref{5bis} and \eqref{1} furnishes 
\begin{equation*}
-\Delta _{p_{i}}z_{i}(x)-(-\Delta _{p_{i}}z_{i,\delta }(x))=\left\{ 
\begin{array}{ll}
0 & \text{in }\Omega \backslash\overline{\Omega }_{\delta }, \\ 
1+\lambda ^{\theta p_{i}}d(x)^{\gamma _{i}} & \text{in }\Omega _{\delta }.%
\end{array}
\right.
\end{equation*}
Due to \eqref{1} again, besides \eqref{emme}, it easily implies 
\begin{equation*}
\begin{split}
& \int_{\Omega}(|\nabla z_{i}|^{p_{i}-2}\nabla z_{i}-|\nabla
z_{i,\delta}|^{p_{i}-2}\nabla z_{i,\delta })\nabla (z_{i}-z_{i,\delta })\,%
\mathrm{d}x \\
& \leq 2\hat M_{i}(1+\lambda ^{\theta p_{i}})\int_{\Omega _{\delta
}}d^{\gamma_{i}}\,\mathrm{d}x,
\end{split}%
\end{equation*}
whence, on account of \cite[Lemma A.0.5]{P}, 
\begin{equation*}
\lim_{\delta \rightarrow 0^{+}}\Vert \nabla z_{i,\delta }-\nabla z_{i}\Vert
_{p_{i}}=0.
\end{equation*}
Observe that $\int_{\Omega _{\delta }}d^{\gamma _{i}}\,\mathrm{d}x<+\infty $%
, as $-1<\gamma _{i}<0$ and so \cite[Lemma]{LM} applies. Since the embedding 
$C^{1,\tau }(\overline{\Omega })\subseteq C^{1}(\overline{\Omega })$ is
compact, up to subsequences, we thus have 
\begin{equation}
\lim_{\delta \rightarrow 0^{+}}\Vert z_{i,\delta }-z_{i}\Vert _{C^{1}(%
\overline{\Omega })}=0.  \label{subseq}
\end{equation}
From \eqref{12*} it follows $k_{i}:=-\max_{\partial \Omega }\frac{\partial
z_{i}}{\partial \eta }>0$ while, by \eqref{subseq}, 
\begin{equation*}
\lim_{\delta \rightarrow 0^{+}}\frac{\partial z_{i,\delta }}{\partial \eta }%
= \frac{\partial z_{i}}{\partial \eta }\;\;\mbox{uniformly in}%
\;\;\partial\Omega .
\end{equation*}
Hence, there exists $\delta _{0}>0$ such that 
\begin{equation*}  \label{firstc}
\frac{\partial z_{i,\delta }}{\partial \eta }<\frac{1}{2}\frac{\partial z_{i}%
}{\partial \eta }\leq -\frac{k_{i}}{2}<0\;\;\text{on}\;\;\partial \Omega
\end{equation*}
for all $\delta <\delta _{0}$. This shows conclusion $(\mathrm{j}_{1})$.

%
%
%
%
%
%
%
%
%
%
%
%
Thanks to Lemma \ref{L1} and \eqref{subseq} we get 
\begin{equation*}  \label{zoverd}
\lim_{\delta \rightarrow 0^{+}}\Big\Vert\frac{z_{i,\delta }}{d}-\frac{z_{i}}{%
d}\Big\Vert_{C^{0}(\overline{\Omega })}=0.
\end{equation*}
Bearing in mind \eqref{12*} one arrives at 
\begin{equation*}
\frac{z_{i,\delta}}{d}>\frac{z_i}{d}-\frac{l}{2}\geq\frac{l}{2}%
\quad\forall\,\delta\in\ ]0,\delta_1[
\end{equation*}
with suitable $\delta_1>0$. Consequently, 
\begin{equation*}
\frac{z_{i,\delta}}{d}>\frac{z_i}{d}-\frac{l}{2}\geq\frac{z_i}{d}-\frac{%
z_{i.\delta}}{d} ,
\end{equation*}
which immediately forces $(\mathrm{j}_{2})$. 
%
%
%
%
%
\end{proof}

Given $\delta ,\lambda >0$, define 
\begin{equation}  \label{2*}
\underline{u}:=\frac{1}{\lambda}\left(z_{1,\delta }-\frac{l\delta }{2}%
\right),\quad \underline{v}:=\frac{1}{\lambda}\left(z_{2,\delta }-\frac{%
l\delta }{2}\right),
\end{equation}
\begin{equation}  \label{32}
\overline{u}:=\lambda ^{p_{1}^{\prime }}\left( z_{1}^{\omega _{1}}-\left(%
\frac{L\delta}{\lambda^\theta}\right)^{\omega_{1}}\right),\quad \overline{v}%
:=\lambda^{p_{2}^{\prime }}\left( z_{2}^{\omega _{2}}-\left(\frac{L\delta}{%
\lambda^\theta}\right)^{\omega_{2}}\right),
\end{equation}
where $\theta $ satisfies \eqref{5} while 
\begin{equation}
\omega _{i}:=1+\frac{\gamma _{i}+1}{p_{i}-1}:=1+\lambda ^{-\theta p_{i}}.
\label{34}
\end{equation}
Via $(\mathrm{j}_{1})$ of Lemma \ref{L3} we obtain 
\begin{equation}  \label{14}
\frac{\partial \underline{u}}{\partial \eta }=\frac{1}{\lambda}\frac{%
\partial z_{1,\delta }}{\partial \eta }<0\quad\text{and}\quad\frac{\partial 
\underline{v}}{\partial \eta }=\frac{1}{\lambda}\frac{\partial z_{2,\delta }%
}{\partial \eta }<0\quad\text{on}\quad\partial \Omega .
\end{equation}
From \eqref{12}--\eqref{12*} it follows 
\begin{equation}  \label{15*}
\overline{u}\leq \lambda ^{p_{1}^{\prime }}(Ld)^{\omega _{1}},\quad 
\overline{v}\leq \lambda ^{p_{2}^{\prime }}(Ld)^{\omega _{2}},
\end{equation}
as well as 
\begin{equation*}
\Vert \nabla \overline{u}\Vert _{\infty }\leq \lambda ^{p_{1}^{\prime }}\hat{%
L}_{1},\quad \Vert \nabla \overline{v}\Vert _{\infty }\leq
\lambda^{p_{2}^{\prime }}\hat{L}_{2},  \label{15}
\end{equation*}
with $\hat{L}_{i}:=\omega _{i}(L|\Omega |)^{\omega _{i}-1}\hat{L},$ $i=1,2$.
Moreover, 
\begin{equation}
\begin{split}
\frac{\partial \overline{u}}{\partial \eta }=\lambda ^{p_{1}^{\prime }}\frac{%
\partial (z_{1}^{\omega _{1}})}{\partial \eta }=\lambda
^{p_{1}^{\prime}}\omega _{1}z_{1}^{\omega _{1}-1}\frac{\partial z_{1}}{%
\partial \eta }=0, \\
\frac{\partial \overline{v}}{\partial \eta }=\lambda ^{p_{2}^{\prime }}\frac{%
\partial (z_{2}^{\omega _{2}})}{\partial \eta }=\lambda
^{p_{2}^{\prime}}\omega _{2}z_{2}^{\omega _{2}-1}\frac{\partial z_{2}}{%
\partial \eta }=0
\end{split}
\label{24}
\end{equation}
on $\partial \Omega $, because $z_{i}$ solves \eqref{5bis} and $\omega
_{i}>1 $, $i=1,2$.

\begin{lemma}
\label{L3} Under \eqref{16}, with a large fixed $\lambda >0,$ one has both $%
\underline{u}\leq \overline{u}$ and $\underline{v}\leq \overline{v}$
provided $\theta >0$ is big enough.
\end{lemma}

\begin{proof}
A direct computation gives 
\begin{equation}
\begin{split}
-\Delta _{p_{1}}\overline{u}& =-\Delta _{p_{1}}(\lambda
^{p_{1}^{\prime}}z_{1}^{\omega _{1}})= \lambda ^{p_{1}}\omega _{1}^{p_{1}-1}%
\left[z_{1}^{(\omega _{1}-1)(p_{1}-1)}\right. \\
& \phantom{PPPPP} \left. -\left( \omega _{1}-1\right) \left(
p_{1}-1\right)z_{1}^{(\omega _{1}-1)(p_{1}-1)-1}|\nabla z_{1}|^{p_{1}}\right]
\\
& =\lambda ^{p_{1}}\omega _{1}^{p_{1}-1}\left[ z_{1}-\left(
\omega_{1}-1\right) \left( p_{1}-1\right) |\nabla z_{1}|^{p_{1}}\right]
z_{1}^{(\omega _{1}-1)(p_{1}-1)-1}.
\end{split}
\label{47}
\end{equation}
Using \eqref{34}, \eqref{12}--\eqref{12*}, and \eqref{5} yields 
\begin{equation}
\begin{split}
& \lambda ^{p_{1}}\omega _{1}^{p_{1}-1}\left[ z_{1}-(\omega_{1}-1)(p_{1}-1)|%
\nabla z_{1}|^{p_{1}}\right] z_{1}^{(\omega_{1}-1)(p_{1}-1)-1} \\
& =\lambda ^{p_{1}}\omega _{1}^{p_{1}-1}\left[ z_{1}-\lambda ^{-\theta
p_{1}}(p_{1}-1)|\nabla z_{1}|^{p_{1}}\right] z_{1}^{(%
\omega_{1}-1)(p_{1}-1)-1} \\
& \geq \lambda ^{p_{1}}\omega _{1}^{p_{1}-1}\left[ ld-\lambda ^{-\theta
p_{1}}(p_{1}-1)\hat{L}^{p_{1}}\right] (Ld_*)^{(\omega _{1}-1)(p_{1}-1)-1} \\
& \geq \lambda ^{p_{1}}\left[ l\delta -\lambda ^{-\theta p_{1}}(p_{1}-1)\hat{%
L}^{p_{1}}\right] (Ld_*)^{(\omega _{1}-1)(p_{1}-1)-1}\geq \lambda
^{-(p_{1}-1)}
\end{split}
\label{48}
\end{equation}
in $\Omega \backslash\overline{\Omega }_{\delta }$ once $\lambda,\theta>0$
are sufficiently large. Here, $d_*:= \max_{\overline{\Omega}}d$. By %
\eqref{34}, \eqref{12}, \eqref{12*}, besides \eqref{16}, we next obtain 
\begin{equation}
\begin{split}
& \lambda ^{p_{1}}\omega _{1}^{p_{1}-1}\left[ z_{1}-(\omega_{1}-1)(p_{1}-1)|%
\nabla z_{1}|^{p_{1}}\right] z_{1}^{(\omega_{1}-1)(p_{1}-1)-1} \\
& =\lambda ^{p_{1}}\omega _{1}^{p_{1}-1}\left[ z_{1}-\lambda ^{-\theta
p_{1}}(p_{1}-1)|\nabla z_{1}|^{p_{1}}\right] z_{1}^{\gamma _{1}} \\
& \geq -\lambda ^{(1-\theta )p_{1}}\omega _{1}^{p_{1}-1}(p_{1}-1)|\nabla
z_{1}|^{p_{1}}z_{1}^{\gamma _{1}} \\
& \geq -\lambda ^{(1-\theta )p_{1}}2^{p_{1}-1}(p_{1}-1)\hat{L}^{p_{1}}
(ld)^{\gamma _{1}}\geq -\lambda ^{(\theta -1)p_{1}+1}d^{\gamma _{1}}
\end{split}
\label{49}
\end{equation}
in $\Omega _{\delta }$, because $\omega _{1}<2$. On the other hand, due to %
\eqref{1}, from \eqref{2*} it follows 
\begin{equation}  \label{50}
\begin{array}{l}
-\Delta _{p_{1}}\underline{u}(x) =\left\{ 
\begin{array}{ll}
\lambda ^{-(p_{1}-1)} & \text{in}\;\;\Omega \backslash \overline{\Omega }%
_{\delta }, \\ 
-\lambda ^{(\theta -1)p_{1}+1}d(x)^{\gamma _{1}} & \text{in}%
\;\;\Omega_{\delta }.%
\end{array}
\right.%
\end{array}%
\end{equation}
Now, gathering \eqref{47}--\eqref{48} and \eqref{49}--\eqref{50} together
one achieves 
\begin{equation*}
-\Delta _{p_{1}}\underline{u}\leq -\Delta _{p_{1}}\overline{u}\,.
\end{equation*}
Since \eqref{2*}--\eqref{32} and the choice of $\lambda $ entail, for any
sufficiently large $\theta$, 
\begin{equation*}
\underline{u}\leq \overline{u}<0\;\;\text{on}\;\;\partial \Omega ,
\end{equation*}
through \cite[Lemma 3.1]{T} we achieve $\underline{u}\leq \overline{u}$ in $%
\Omega $, as desired. A quite similar argument ensures that $\underline{v}%
\leq \overline{v}$.
\end{proof}

\begin{remark}
Carefully reading this proof reveals that the constant $\theta$ in \eqref{5}
can be precisely estimated.
\end{remark}

We will posit the hypotheses below.

\begin{itemize}
\item[$(\mathrm{h}_{1})$] There exist $\alpha _{i},\beta _{i},M_{i}>0$, $%
i=1,2$, such that 
\begin{equation}
q:=\alpha _{i}p_{1}^{\prime }+\beta _{i}p_{2}^{\prime }<1  \label{c}
\end{equation}
and, moreover, 
\begin{equation*}
\begin{array}{l}
|f(x,s,t)|\leq M_{1}(1+|s|^{\alpha _{1}})(1+|t|^{\beta _{1}}), \\ 
|g(x,s,t)|\leq M_{2}(1+|s|^{\alpha _{2}})(1+|t|^{\beta _{2}})%
\end{array}%
\end{equation*}
for all $(x,s,t)\in \Omega \times \mathbb{R}^{2}$.

\item[$(\mathrm{h}_2)$] With appropriate $m_i,\rho_i>0$, $i=1,2$, one has 
\begin{equation}  \label{ha}
\lim_{|s|\to 0}\inf \left\{ f(x,s,t):-\rho_1\leq t \right\}>-m_{1},
\end{equation}
\begin{equation}  \label{hb}
\lim_{|t|\to 0}\inf\left\{ g(x,s,t):-\rho_2\leq s\right\}>-m_{2}
\end{equation}
uniformly in $x\in \Omega$.
\end{itemize}

\begin{theorem}
\label{T1}  Let $\gamma_i$, $i=1,2$, be given by \eqref{16}, with a large
fixed $\lambda >0,$ and let $(\mathrm{h}_1)$--$(\mathrm{h}_2)$ be satisfied.
Then problem \eqref{p} admits a nodal solution $(u_{0},v_{0})\in
W_{b}^{1,p_{1}}(\Omega )\times W_{b}^{1,p_{2}}(\Omega )$ provided $\theta >0$
is big enough. Further, both $u_{0}(x)$ and $v_{0}(x)$ are negative once $%
d(x)\rightarrow 0$.
\end{theorem}

\begin{proof}
Assumption $(\mathrm{h}_{1})$ evidently forces $(\mathrm{a}_{1})$ of Section %
\ref{S3}, while Lemma \ref{L3} gives $\underline{u} \leq\overline{u}$ and $%
\underline{v}\leq\overline{v}$. Fix $\delta >0$ fulfilling 
\begin{equation*}
\frac{l\delta }{2\lambda }<\min \{\rho _{1},\rho _{2}\}.
\end{equation*}
We claim that \eqref{c2} holds. To see this, pick $(u,v)\in
W^{1,p_{1}}(\Omega )\times W^{1,p_{2}}(\Omega)$ within $[\underline{u},%
\overline{u}]\times[\underline{v},\overline{v}]$. Due to \eqref{12*}, Lemma %
\ref{L3} yields 
\begin{equation}  \label{nn1}
\min \{u,v\}\geq \min \{\underline{u},\underline{v}\}\geq \frac{l(d-\delta )%
}{2\lambda }\geq-\frac{l\delta }{2\lambda} >-\max \{\rho _{1},\rho _{2}\}.
\end{equation}
From \eqref{15*} it follows 
\begin{equation}  \label{nn2}
u\leq \overline{u}\leq \lambda ^{p_{1}^{\prime }}(Ld)^{\omega _{1}},\quad
v\leq \overline{v}\leq \lambda ^{p_{2}^{\prime }}(Ld)^{\omega_{2}}.
\end{equation}
Hence, on account of \eqref{nn1}--\eqref{nn2}, 
\begin{equation}
-\rho_1<-\frac{l\delta }{2\lambda }\leq \frac{l(d-\delta )}{2\lambda }\leq
u\leq \lambda ^{p_{1}^{\prime }}(Ld)^{\omega _{1}}\leq \lambda
^{p_{1}^{\prime }}C_{1},  \label{n2}
\end{equation}
\begin{equation*}
-\rho_2<-\frac{l\delta }{2\lambda }\leq \frac{l(d-\delta )}{2\lambda }\leq
v\leq \lambda ^{p_{2}^{\prime }}(Ld)^{\omega _{2}}\leq \lambda
^{p_{1}^{\prime }}C_{2},
\end{equation*}
with $C_i:=(L\ d_{\ast})^{w_i}$, $i=1,2$, and $d_*:=\max_{\overline{\Omega}}d
$. Now, \eqref{ha} yields $\bar{\eta}_{m_1}>0$ such that 
\begin{equation}
f(x,s,t)>-m_{1}  \label{3}
\end{equation}
provided $x\in \Omega $, $|s|<\bar{\eta}_{m_{1}}$, $\frac{l(d(x)-\delta )}{%
2\lambda }\leq t\leq \lambda ^{p_{2}^{\prime }}(Ld(x))^{\omega _{2}}$.
Likewise, via \eqref{hb} we obtain 
\begin{equation}
g(x,s,t)>-m_{2}  \label{3*}
\end{equation}
once $x\in \Omega $, $\frac{l(d(x)-\delta )}{2\lambda}\leq s\leq
\lambda^{p_1^{\prime }}(Ld(x))^{\omega_1}$, $|t|<\bar{\eta}_{m_2}$.\newline
Pick any $x\in \Omega\setminus\overline{\Omega }_{\delta }$. By \eqref{2}, %
\eqref{50}, and \eqref{n2}--\eqref{3} one has, after increasing $\lambda $
when necessary, 
\begin{equation}
\begin{split}
-\Delta _{p_{1}}\underline{u}(x)-\lambda
h_{1}(x)=\lambda^{-(p_{1}-1)}-\lambda h_{1}(x) & =\lambda
^{-(p_{1}-1)}-\lambda d(x)^{\gamma_{1}} \\
& <-m_{1}<f(x,\underline{u}(x),v(x)).
\end{split}
\label{n4}
\end{equation}
If $x\in \Omega _{\delta }$ then, thanks to \eqref{2}, \eqref{5}, \eqref{50}%
, and \eqref{n2}--\eqref{3}, 
\begin{equation}
\begin{split}
-\Delta _{p_{1}}\underline{u}(x)-\lambda h_{1}(x) &
=-\lambda^{-(p_{1}-1)}\lambda ^{\theta p_{1}}d(x)^{\gamma _{1}}-\lambda
h_{1}(x) \\
& \leq \left[ \lambda -\lambda ^{(\theta -1)p_{1}+1}\right] d(x)^{\gamma_{1}}
\\
& <-m_{1}<f(x,\underline{u}(x),v(x))
\end{split}
\label{n4*}
\end{equation}
for all $\lambda,\theta >0$ sufficiently large. Gathering \eqref{1}, %
\eqref{2*}, \eqref{n4}--\eqref{n4*} together we get 
\begin{equation}
-\Delta _{p_{1}}\underline{u}\leq f(\cdot,\underline{u},v)+\lambda h_{1}.
\label{46}
\end{equation}
A quite similar argument, which employs \eqref{3*} instead of \eqref{3},
furnishes 
\begin{equation}
-\Delta _{p_{2}}\underline{v}\leq g(\cdot,u,\underline{v})+\lambda h_{2}.
\label{46*}
\end{equation}
Finally, test \eqref{46}--\eqref{46*} with $(\varphi ,\psi )\in
W_{+}^{1,p_{1}}(\Omega )\times W_{+}^{1,p_{2}}(\Omega )$ and recall %
\eqref{14}, besides Green's formula \cite{CF}, to arrive at 
\begin{equation*}
\begin{split}
\int_{\Omega }|\nabla \underline{u}|^{p_{1}-2}\nabla \underline{u}%
\nabla\varphi \,\mathrm{d}x & \leq \int_{\Omega }|\nabla \underline{u}%
|^{p_{1}-2}\nabla \underline{u}\nabla \varphi \,\mathrm{d}x-\left\langle 
\frac{\partial \underline{u}}{\partial \eta _{p_{1}}},\gamma
_{0}(\varphi)\right\rangle _{\partial \Omega } \\
& =\int_{\Omega }-\Delta _{p_{1}}\underline{u}\ \varphi \,\mathrm{d}x \\
& \leq \int_{\Omega }\left( f(\cdot ,\underline{u},v)+\lambda h_{1}\right)
\varphi \,\mathrm{d}x,
\end{split}%
\end{equation*}
\begin{equation*}
\begin{split}
\int_{\Omega }|\nabla \underline{v}|^{p_{2}-2}\nabla \underline{v}\nabla\psi
\,\mathrm{d}x & \leq \int_{\Omega }|\nabla\underline{v}|^{p_{2}-2}\nabla 
\underline{v}\nabla \psi \,\mathrm{d}x-\left\langle \frac{\partial 
\underline{v}}{\partial \eta _{p_{2}}},\gamma _{0}(\psi)\right\rangle
_{\partial \Omega } \\
& =\int_{\Omega }-\Delta _{p_{2}}\underline{v}\ \psi \,\mathrm{d}x \\
& \leq \int_{\Omega }\left( g(\cdot ,u,\underline{v})+\lambda
h_{2}\right)\psi \,\mathrm{d}x,
\end{split}%
\end{equation*}
because $\gamma _{0}(w)\geq 0$ whatever $w\in W_{+}^{1,p_{i}}(\Omega )$, see 
\cite[p. 35]{CLM}. Here, $\gamma _{0}$ is the trace operator on $%
\partial\Omega $, 
\begin{equation}
\frac{\partial w}{\partial \eta _{p_{i}}}:=|\nabla w|^{p_{i}-2}\frac{%
\partial w}{\partial \eta } \quad \forall \,w\in W^{1,p_{i}}(\Omega )\cap
C^{1}(\overline{\Omega }),  \label{conode}
\end{equation}
while $\left\langle \cdot ,\cdot \right\rangle _{\partial \Omega }$ denotes
the duality brackets for the pair 
\begin{equation*}
(W^{1/p_{i}^{\prime },p_{i}}(\partial \Omega ),W^{-1/p_{i}^{\prime
},p_{i}^{\prime }}(\partial \Omega )).
\end{equation*}
Let us next show that the functions $\overline{u}$ and $\overline{v}$ given
by \eqref{32} satisfy \eqref{c3}. With this aim, pick $(u,v)\in
W^{1,p_{1}}(\Omega )\times W^{1,p_{2}}(\Omega )$ such that $\underline{u}%
\leq u\leq \overline{u}$, $\underline{v} \leq v\leq \overline{v}$. From %
\eqref{n2} and \eqref{c} it follows 
\begin{equation*}  \label{26}
\begin{split}
f(\cdot,\overline{u},v) & \leq M_1\left(1+|\overline{u}|^{\alpha_1}\right)%
\left(1+|v|^{\beta_1}\right) \\
& \leq M_1\left(1+C_1^{\alpha_1}\lambda^{p_1^\prime\alpha_1}\right)
\left(1+C_2^{\beta_1}\lambda^{p_2^\prime\beta_2}\right) \\
& \leq 2M_1C_1^{\alpha_1}C_2^{\beta_1}\lambda^q,
\end{split}%
\end{equation*}
provided $\lambda$ is big enough. Hence, 
\begin{equation}  \label{cla1}
f(\cdot,\overline{u},v)\leq C \lambda^q,
\end{equation}
where $C:= 2M_1C_1^{\alpha_1}C_2^{\beta_1}$. By \eqref{48}--\eqref{49} one
has 
\begin{equation}  \label{422425bis}
-\Delta _{p_1}\overline{u}\geq \left\{ 
\begin{array}{ll}
\lambda^{p_1}\left[ l\delta -\lambda^{-\theta p_1}\left(p_1-1\right) \hat{L}%
^{p_1}\right] L^{\gamma _{1}}d^{\gamma_1} & \text{in}\;\;\Omega \backslash 
\overline{\Omega}_{\delta}, \\ 
-\lambda ^{(1-\theta )p_1}2^{p_1-1}(p_1-1)\hat{L}^{p_1}L^{\gamma_{1}}d^{%
\gamma_{1}} & \text{in}\;\;\Omega_{\delta }.%
\end{array}
\right.
\end{equation}
Moreover, 
\begin{equation}  \label{cla2}
\lambda^{p_1-1}\left[ l\delta -\lambda^{-\theta p_1}\left(p_1-1\right) \hat{L%
}^{p_1}\right] L^{\gamma_1}d(x)^{\gamma_1} \geq C+d(x)^{\gamma
_{1}}\quad\forall\, x\in\Omega\setminus\overline{\Omega }_{\delta },
\end{equation}
In fact, after increasing $\lambda $ and $\theta $ if necessary, we achieve 
\begin{equation*}
\begin{split}
\lambda^{p_1}\left[ l\delta -\lambda^{-\theta p_1}\left(p_1-1\right) \hat{L}%
^{p_{1}}\right] L^{\gamma_1}d(x)^{\gamma_1} & \geq \frac{l\delta }{2}\hat{L}%
^{p_{1}}d_{\ast }^{\gamma _{1}}\lambda ^{p_{1}} \\
& \geq \lambda (C+\delta^{\gamma_1}) \\
& \geq\lambda (C+d(x)^{\gamma_1}),
\end{split}
\quad x\in\Omega\setminus\overline{\Omega }_\delta,
\end{equation*}
with $d_{\ast }:=\max_{\overline{\Omega}}d$. Thus, \eqref{422425bis}--%
\eqref{cla2} and \eqref{cla1} yield 
\begin{equation*}
-\Delta_{p_1}\overline{u}\geq f(\cdot,\overline{u},v)+\lambda h_1\;\;\text{in%
}\;\;\Omega\setminus\overline{\Omega}_\delta.
\end{equation*}
Let now $x\in {\Omega }_{\delta }$. Inequalities \eqref{cla1}--%
\eqref{422425bis} entail 
\begin{equation*}
\begin{split}
-\Delta _{p_1}\overline{u}(x)+\lambda d(x)^{\gamma _{1}} & \geq
d(x)^{\gamma_1}\left( \lambda-\lambda ^{(1-\theta)p_1}2^{p_1-1}(p_1-1)\hat{L}%
^{p_1}L^{\gamma_1}\right) \\
& \geq \delta ^{\gamma_1}\left(\lambda
-\lambda^{(1-\theta)p_1}2^{p_{1}-1}(p_{1}-1)\hat{L}^{p_1}L^{\gamma_{1}}%
\right) \\
& \geq \delta ^{\gamma _{1}}\frac{\lambda }{2}\geq C\lambda^{q}\geq f(x,%
\overline{u}(x),v(x))
\end{split}%
\end{equation*}
for any $\lambda,\theta>0$\ big enough, that is 
\begin{equation*}
-\Delta _{p_{1}}\overline{u}\geq f(\cdot,\overline{u},v)+\lambda h_{1}\;\;%
\text{in}\;\;\Omega _{\delta }.
\end{equation*}
Summing up, 
\begin{equation*}
-\Delta _{p_{1}}\overline{u}\geq f(\cdot,\overline{u},v)+\lambda h_{1}\;\;%
\text{on the whole}\;\;\Omega.
\end{equation*}
Finally, test with $\varphi \in W_{+}^{1,p_{1}}(\Omega )$ and recall %
\eqref{24}, besides \eqref{conode}, to get 
\begin{equation*}
\begin{split}
\int_{\Omega }|\nabla \overline{u}|^{p_{1}-2}\nabla \overline{u}%
\nabla\varphi \,\mathrm{d}x & =\int_{\Omega }|\nabla \overline{u}%
|^{p_{1}-2}\nabla \overline{u}\nabla \varphi \,\mathrm{d}x -\left\langle 
\frac{\partial \overline{u}}{\partial \eta _{p_{1}}},\gamma _{0}(\varphi
)\right\rangle_{\partial \Omega } \\
& \geq \int_{\Omega }\left( f(\cdot ,\overline{u},v)+\lambda
h_{1}\right)\varphi \,\mathrm{d}x,
\end{split}%
\end{equation*}
as desired. Analogously, one has 
\begin{equation*}
\int_{\Omega }|\nabla \overline{v}|^{p_{2}-2}\nabla \overline{v}\nabla \psi\,%
\mathrm{d}x \geq \int_\Omega\left( g(\cdot,u,\overline{v})+\lambda
h_2\right) \psi \,\mathrm{d}x\;\;\forall \,\psi \in W_{+}^{1,p_1}(\Omega).
\end{equation*}
Therefore, $(\underline{u},\underline{v})$ and $(\overline{u},\overline{v})$
satisfy assumption $(\mathrm{a}_{2})$, whence Theorem \ref{T2} can be
applied, and there exists a solution $(u_{0},v_{0})\in W_b^{1,p_1}(\Omega
)\times W_b^{1,p_2}(\Omega )$ of problem \eqref{p} such that 
\begin{equation}
\underline{u}\leq u_{0}\leq \overline{u},\quad \underline{v}\leq v_{0}\leq 
\overline{v}.  \label{22}
\end{equation}
Moreover, $(u_{0},v_{0})$ is nodal. In fact, through \eqref{32} and %
\eqref{12*} we obtain 
\begin{equation*}
\begin{split}
\overline{u}=\lambda ^{p_1^{\prime }}\left( {z_1}^{\omega_1}-(\lambda^{-%
\theta}L\delta )^{\omega _{1}}\right) & \leq \lambda^{p_{1}^{\prime }}\left[
(Ld)^{\omega _{1}}-(\lambda ^{-\theta }L\delta)^{\omega _{1}}\right] \\
& =\lambda^{p_{1}^{\prime }}L^{\omega_{1}}\left( d^{\omega_{1}}-(\lambda
^{-\theta }\delta )^{\omega _{1}}\right) ,
\end{split}%
\end{equation*}
\begin{equation*}
\begin{split}
\overline{v}=\lambda ^{p_{2}^{\prime }}\left( z_{2}^{\omega_{2}}-(\lambda
^{-\theta }L\delta )^{\omega _{2}}\right) & \leq \lambda^{p_{2}^{\prime }}%
\left[ (Ld)^{\omega _{2}}-(\lambda ^{-\theta }L\delta)^{\omega _{2}}\right]
\\
& =\lambda ^{p_{2}^{\prime }}L^{\omega _{2}}\left( d^{\omega_{2}}-(\lambda
^{-\theta }\delta )^{\omega _{2}}\right),
\end{split}%
\end{equation*}
which actually means 
\begin{equation}  \label{n5}
\max \{\overline{u}(x),\overline{v}(x)\}<0\;\;\text{provided}%
\;\;d(x)<\lambda ^{-\theta }\delta .
\end{equation}
Gathering \eqref{2*} and \eqref{12*} together yields 
\begin{equation*}
\begin{split}
\underline{u}& =\frac{1}{\lambda}\left( z_{1,\delta }-\frac{l\delta }{2}%
\right) \geq \frac{l}{2\lambda }(d-\delta ), \\
\underline{v}& =\frac{1}{\lambda}\left( z_{2,\delta }-\frac{l\delta }{2}%
\right) \geq \frac{l}{2\lambda }(d-\delta ).
\end{split}%
\end{equation*}
Consequently, 
\begin{equation}
\min \{\underline{u}(x),\underline{v}(x)\}>0\;\;\text{as soon as}%
\;\;d(x)>\delta .  \label{n6}
\end{equation}
On account of \eqref{22}--\eqref{n6}, the conclusion follows.
\end{proof}

Finding positive solutions is a much simpler matter.

\begin{theorem}
\label{T3} If $(\mathrm{h}_1)$--$(\mathrm{h}_2)$ hold and $\gamma_1,\gamma_2$
are given by \eqref{16}, with $\lambda,\theta>0$ sufficiently large, then %
\eqref{p} admits a solution $(u^\ast,v^\ast)\in W^{1,p_1}_b (\Omega)\times
W^{1,p_2}_b(\Omega )$ such that 
\begin{equation}  \label{4}
\min\{ u^\ast,v^\ast\}\geq c\,d
\end{equation}
for some $c>0$.
\end{theorem}

\begin{proof}
Keep the same notation as before and define 
\begin{equation*}
(\underline{u}^{\ast},\underline{v}^{\ast}):=\frac{1}{\lambda}%
(z_{1,\delta},z_{2,\delta}),\quad (\overline{u}^\ast,\overline{v}%
^\ast):=(\lambda^{p_1^\prime}z_1^{\omega_1},\lambda^{p_2^\prime}z_2^{%
\omega_2}).
\end{equation*}
The arguments exploited in the proof of Theorem \ref{T1} ensure here that $( 
\underline{u}^\ast,\underline{v}^\ast)$ and $(\overline{u}^\ast, \overline{v}%
^\ast)$ fulfill $(\mathrm{a}_2)$ provided $\lambda,\theta>0$ are big enough.
So, thanks to Theorem \ref{T2}, we obtain a solution $(u^\ast,v^\ast)\in
W^{1,p_1}_b(\Omega )\times W^{1,p_2}_b(\Omega )$ of \eqref{p} lying in $[%
\underline{u}^{\ast },\overline{u}^{\ast }]\times \lbrack \underline{v}%
^{\ast },\overline{v}^{\ast }]$. Finally, Lemma \ref{L3} and \eqref{12*}
easily entail \eqref{4}.
\end{proof}

\vskip0.5cm \centerline{\textsc{Acknowledgements}} \smallskip The authors
thank S.J.N. Mosconi for pointing out Lemma 1 and helping to prove Lemma 2.

This work was partially performed when the third-named author visited
Catania and Reggio Calabria universities, to which he is grateful for the
kind hospitality.

\end{document}